\documentclass{amsart}

\usepackage{amssymb, amsmath}
\usepackage{mathrsfs, color}
\usepackage{amscd}
\usepackage{verbatim}

\usepackage{mathtools}

\usepackage{enumerate}

\allowdisplaybreaks

\theoremstyle{plain}
\newtheorem{theorem}{Theorem}[section]
\theoremstyle{remark}
\newtheorem{remark}[theorem]{Remark}
\newtheorem{example}[theorem]{Example}

\theoremstyle{plain}
\newtheorem{corollary}[theorem]{Corollary}
\newtheorem{lemma}[theorem]{Lemma}
\newtheorem{proposition}[theorem]{Proposition}
\newtheorem{definition}[theorem]{Definition}

\numberwithin{equation}{section}

\newcommand*{\lcdot}{\raisebox{-0.5ex}{\scalebox{2}{$\cdot$}}}
\def\avint_#1{\mathchoice%
      {\mathop{\kern 0.2em\vrule width 0.6em height 0.69678ex depth -0.58065ex
              \kern -0.8em \intop}\nolimits_{\kern -0.4em#1}}%
      {\mathop{\kern 0.1em\vrule width 0.5em height 0.69678ex depth -0.60387ex
              \kern -0.6em \intop}\nolimits_{#1}}%
      {\mathop{\kern 0.1em\vrule width 0.5em height 0.69678ex depth -0.60387ex
              \kern -0.6em \intop}\nolimits_{#1}}%
      {\mathop{\kern 0.1em\vrule width 0.5em height 0.69678ex depth -0.60387ex
              \kern -0.6em \intop}\nolimits_{#1}}}

\newcommand{\N}{\ensuremath{\mathbb{N}}}

\newcommand{\R}{\ensuremath{\mathbb{R}}}

\newcommand{\K}{\ensuremath{\mathcal{K}}}

\newcommand{\I}{\ensuremath{\mathcal{I}}}
\newcommand{\J}{\ensuremath{\mathcal{J}}}
\newcommand{\B}{\ensuremath{\mathcal{B}}}
\newcommand{\calL}{\ensuremath{\mathcal{L}}}
\newcommand{\T}{\ensuremath{\mathscr{T}}}
\newcommand{\scrS}{\ensuremath{\mathscr{S}}}

\newcommand{\Rr}{\ensuremath{\mathcal{R}}}

\DeclarePairedDelimiter\abs{\lvert}{\rvert}

\DeclarePairedDelimiter\cbrace\{\}
\DeclarePairedDelimiter\ha()
\DeclarePairedDelimiter{\ip}\langle\rangle
\DeclarePairedDelimiter{\nrm}\lVert\rVert

\newcommand{\nrmn}[1]{\biggl\|#1\biggr\|}

\newcommand{\dd}{\hspace{2pt}\mathrm{d}}
\newcommand{\vect}[1]{{\overline{#1}}}

\newcommand{\lb}{\langle}
\newcommand{\rb}{\rangle}

\renewcommand{\l}{\ensuremath{\ell}}

\begin{document}

\author{Chiara Gallarati}
\author{Emiel Lorist}
\author{Mark Veraar}
\email{C.Gallarati@tudelft.nl}
\email{EmielLorist@gmail.com}
\email{M.C.Veraar@tudelft.nl}
\address{Delft Institute of Applied Mathematics\\
Delft University of Technology \\ P.O. Box 5031\\ 2600 GA Delft\\The
Netherlands}

\date\today

\title{On the $\ell^s$-boundedness of a family of integral operators}

\begin{abstract}
In this paper we prove an $\ell^s$-boundedness result for integral operators with operator-valued kernels.
The proofs are based on extrapolation techniques with weights due to Rubio de Francia. The results will be applied by the first and third author in a subsequent paper where a new approach to maximal $L^p$-regularity for parabolic problems with time-dependent generator is developed.
\end{abstract}




\thanks{The first and third author are supported by Vrije Competitie subsidy 613.001.206 and Vidi subsidy 639.032.427 of the Netherlands Organisation for Scientific Research (NWO)}

\maketitle

\section{Introduction}

In the influential work \cite{Weis01a, We}, Weis has found a characterization of maximal $L^p$-regularity in terms of $\Rr$-sectoriality, which stands for $\Rr$-boundedness of a family of resolvents on a sector.
The definition of $\Rr$-boundedness is given in Definition \ref{def:Rbdd}. It is a random boundedness condition on a family of operators which is a strengthening of uniform boundedness. Maximal regularity of solution to PDEs is important to know as it provides a tool to solve nonlinear PDEs using linearization techniques (see \cite{ClLi, Lun, Pruss02}). An overview on recent developments on maximal $L^p$-regularity can be found in \cite{DHP, KunstWeis}. Maximal $L^p$-regularity means that for all $f\in L^p(0,T;X)$, where $X$ is a Banach space, the solution $u$ of the evolution problem 
\begin{equation}\label{eq:Cauchy}
  \begin{cases}
    u'(t)&=A u(t)+f(t), \ \ t\in (0,T) \\
    u(0)&=0
  \end{cases}
\end{equation}
has the ``maximal'' regularity in the sense that $u',Au$ are both in $L^p(0,T;X)$. Using a mild formulation one sees that to prove maximal $L^p$-regularity one needs to bound a singular integral with operator-valued kernel $A e^{(t-s)A}$.

In \cite{GV} the first and third author have developed a new approach to maximal $L^p$-regularity for the case that the operator $A$ in \eqref{eq:Cauchy} depends on time in a measurable way. In this new approach $\Rr$-boundedness plays a central r\^ole again. Namely, the $\Rr$-boundedness of the family of integral operators $\{I_{k}: k\in \K\}\subseteq L^p(\R;X)$ is required in the proofs. Here $I_k$ is defined by
\begin{equation}\label{eq:inkTintro}
(I_{k} f)(t) = \int_{\R} k(t-s) T(t,s) f(s)\, ds,
\end{equation}
where $T(t,s)\in \calL(X)$ is a two-parameter evolution family and $\K$ is the class of kernels which satisfy $|k|*f\leq Mf$ for $f:\R\to \R_+$ simple and where $M$ is the Hardy-Littlewood maximal operator. For evolution families one usually sets $T(t,s) = 0$ if $t<s$. 

In this paper we give a class of examples for which we can prove the $\Rr$-boundedness of $\{I_{k}: k\in \K\}$. We now state a special case of our main result. It is valid for general families of operators $\{T(t,s):-\infty<s\leq t<\infty\}\subseteq \calL(L^q(\Omega,w))$. We will not use any regularity conditions for $(t,s)\mapsto T(t,s)$ below.

\begin{theorem}\label{thm:intro}
Let $\Omega\subseteq \R^d$ be an open set. Let $p,q\in (1, \infty)$. Assume that for all $A_q$-weights $w$,
\begin{equation}\label{eq:weightedcond}
\|T(t,s)\|_{\calL(L^q(\Omega,w))}\leq C, \ \ s,t\in\R,
\end{equation}
where $C$ depends on the $A_q$-constant of $w$ in a consistent way. Then the family of integral operators
$\{I_{k}: k\in \K\}\subseteq \calL(L^p(\R;L^q(\Omega)))$ as defined in \eqref{eq:inkTintro} is $\Rr$-bounded.
\end{theorem}
In the setting where $T(t,s) = e^{(t-s)A}$ where $A$ is as in \eqref{eq:Cauchy}, the condition \eqref{eq:weightedcond} also appears in \cite{Frohlich2} and \cite{HHH, HH} in order to obtain $\Rr$-sectoriality of $A$. There \eqref{eq:weightedcond} is checked by using Calder\'on-Zygmund and Fourier multiplier theory. Examples of such results for two-parameter evolution families will be given in \cite{GV}. 

As a consequence of the Kahane-Khintchine inequality (see Remark \ref{rem:ell2R}) one can see that in standard spaces such as $L^p$-spaces, $\Rr$-boundedness is equivalent to so-called $\ell^2$-boundedness. The latter is a special case of $\ell^s$-boundedness (see Definition \ref{def:ells}). In $L^p$-spaces this boils down to classical $L^p(\ell^s)$-estimates from harmonic analysis (see \cite{GrafakosC, GrafakosM}, \cite[Chapter V]{Rubioboek} and \cite[Chapter 3]{CMP}).
It follows from the work of Rubio de Francia (see \cite{Rubio82, Rubio83, Rubio84} and \cite{Rubioboek}) that $L^p(\ell^s)$-estimates are strongly connected to estimates in weighted $L^p$-spaces.

To prove Theorem \ref{thm:intro} we apply weighted techniques of Rubio de Francia. Without additional effort we actually prove the more general Corollary \ref{cor:weightedells}, which states that the family of integral operators on $L^p(v,L^q(w))$ is $\ell^s$-bounded for all $p,q,s\in (1, \infty)$ and for arbitrary $A_p$-weights $v$ and $A_q$-weights $w$. Both the modern extrapolation methods with $A_q$-weights as explained in the book of Cruz-Uribe, Martell and P{\'e}rez \cite{CMP} and the factorization techniques of Rubio de Francia (see \cite[Theorem VI.5.2]{Rubioboek} or \cite[Theorem 9.5.8]{GrafakosM}), play a crucial r\^ole in our work.
It is unclear how to apply the extrapolation techniques of \cite{CMP} to the inner space $L^q$ directly, but it does play a r\^ole in our proofs for the outer space $L^p$. The factorization methods of Rubio de Francia enable us to deal with the inner spaces (see the proof of Proposition \ref{prop:mainprel}).

In the literature there are many more $\Rr$-boundedness results for integral operators (e.g. \cite[Section 6]{DDHPV}, \cite[Proposition 3.3 and Theorem 4.12]{DHP}, \cite{GW03}, \cite[Section 3]{HaKu}, \cite[Section 4]{HytVer}, \cite[Chapter 2]{KunstWeis}). However, it seems they are of a different nature and cannot be used to prove Theorems \ref{thm:intro}, \ref{thm:mainnew} and Corollary \ref{cor:weightedells}.

Throughout this paper we will write $\B(X)$ for the space of all bounded operators on a Banach space $X$ and denote the corresponding norm as $\nrm{\cdot}_{\B(X)}$. Let $\calL(X)\subseteq \B(X)$ denote the subspace of all bounded {\em linear} operators. For $p\in [1, \infty]$ we let $p'\in [1, \infty]$ be such that $\frac{1}{p} + \frac{1}{p'} = 1$.

\medskip

\noindent \textbf{Acknowledgement} The authors thank the referees for helpful comments.

\section{Extrapolation and weights}

\subsection{Preliminaries on weights}

First we will introduce Muckenhoupt weights and state some of their properties. Details can be found in \cite[Chapter 9]{GrafakosM} and \cite[Chapter V]{Stein:harmonic}.

A weight is a locally integrable function on $\R^d$ with $w(x)\in (0,\infty)$ for almost every $x \in \R^d$. The space $L^p(\R^d,w)$ is defined as all measurable functions $f$ with
\begin{equation*}
  \nrm{f}_{L^p(\R^d,w)}=\ha*{\int_{\R^d}\abs{f}^p w \dd \mu}^\frac{1}{p}<\infty.
\end{equation*}

With this notion of weights and weighted $L^p$-spaces we can define the class of Muckenhoupt weights $A_{p}$ for all $p \in (1,\infty)$ for a fixed dimension $d \in \N$. Let $\avint_{Q} = \frac{1}{|Q|}\int_{Q}$. For $p\in (1, \infty)$ a weight $w$ is said to be an {\em $A_{p}$-weight} if
\begin{equation*}
   [w]_{A_{p}}=\sup_{Q} \avint_Q w(x) \dd x \ha*{\avint_Q w(x)^{-\frac{1}{p-1}}\dd x }^{p-1}<\infty,
\end{equation*}
where the supremum is taken over all cubes $Q\subseteq \R^d$ with axes parallel to the coordinate axes. The extended real number $[w]_{A_{p}}$ is called the {\em $A_p$-constant}.

Recall that $w \in A_{p}$ if and only if the Hardy-Littlewood maximal operator $M$ is bounded on $L^p(\R^d,w)$. The Hardy-Littlewood maximal operator is defined as
\begin{equation*}
  M(f)(x)=\sup_{Q \ni x}\avint_{Q}\abs{f(y)}\dd y, \ \ \  f \in L^p(\R^d,w)
  \end{equation*}
with $Q$ ranging over all cubes in $\R^d$ with axes parallel to the coordinate axes.

Next we will summarize a few basic properties of weights which we will need.
The proofs can be found in \cite[Theorems 9.1.9 and 9.2.5]{GrafakosM}, \cite[Theorem 9.2.5 and Exercise 9.2.4]{GrafakosM}, \cite[Proposition 9.1.5]{GrafakosM}.
\begin{proposition}
\label{prop:propweight}
  Let $w \in A_{p}$ for some $p \in [1,\infty)$. Then we have
  \begin{enumerate}
    \item\label{it:w2} If $p \in (1,\infty)$ then $w^{-\frac{1}{p-1}} \in A_{p'}$ with
$[w^{-\frac{1}{p-1}}]_{A_{p'}} = [w]_{A_{p}}^{\frac{1}{p-1}}$.
    \item\label{it:w4} For every $p \in (1,\infty)$ and $\kappa>1$ there is a constant $\sigma = \sigma_{p,\kappa,d}\in (1, p)$ and a constant $C_{p,d,\kappa}>1$ such that $[w]_{A_{\frac{p}{\sigma}}}\leq C_{p,\kappa,d}$ whenever $[w]_{A_p}\leq \kappa$. Moreover, $\kappa\mapsto \sigma_{p,\kappa,d}$ and $\kappa\mapsto C_{p,\kappa,d}$ can be chosen to be decreasing and increasing, respectively.
    \item\label{it:w5} $A_{p} \subseteq A_{q}$ and $[w]_{A_{q}}\leq [w]_{A_{p}}$ if $q >p$.
    \item\label{it:w7}  For $p\in (1, \infty)$, there exists a constant $C_{p,d}$ such that
    \[\nrm{M}_{\B(L^p(\R^d,w))}\leq C_{p,d}\cdot [w]_{A_{p}}^{\frac{1}{p-1}}.\]
  \end{enumerate}
\end{proposition}

%

%

\subsection{Extrapolation}

The celebrated result of Rubio de Francia (see \cite{Rubio82, Rubio83, Rubio84}, \cite[Chapter IV]{Rubioboek}) allows one to extrapolate from weighted $L^p$-estimates for a single $p$ to weighted $L^q$-estimates for all $q$. The proofs and statement have been considerably simplified and clarified in \cite{CMP} and can be formulated as follows (see \cite[Theorem 3.9]{CMP}).
\begin{theorem}\label{thm:baseweight}
Let $f,g:\R^d \to \R_+$ be a pair of nonnegative, measurable functions and suppose that for some $p_0\in(1,\infty)$ there exists an increasing function $\alpha$ on $\R_+$ such that for all $w_0 \in A_{p_0}$
\begin{equation*}
\nrm{ f}_{L^{p_0}(\R^d,w_0)} \leq \alpha([w_0]_{A_{p_0}}) \nrm{ g}_{L^{p_0}(\R^d,w_0)}.
\end{equation*}
Then for all $p \in (1,\infty)$ there is a constant $c_{p,d}$ s.t.\ for all $w \in A_{p}$,
\begin{equation*}
\nrm{f}_{L^p(\R^d,w)}\leq 4 \alpha\Big(c_{p,d} [w]_{A_{p}}^{\frac{p_0-1}{p-1}+1}\Big) \nrm{g}_{L^p(\R^d,w)}.
\end{equation*}
\end{theorem}
Note that for certain weights the above $L^p$-norms are allowed to be infinite.
Estimates as in the above result with increasing function $\alpha$ will appear frequently.
In this situation we say that
\[\|f\|_{L^{p_0}(\R^d,w_0)} \leq C\|g\|_{L^{p_0}(\R^d,w_0)}\]
with an {\em $A_{p_0}$-consistent} constant $C$. This means that for two weights $w_0, w_1\in A_p$ we have $C([w_0]_{A_p})\leq C([w_1]_{A_p})$ whenever $[w_0]_{A_p}\leq [w_1]_{A_p}$. Note that the $L^p$-estimate obtained in Theorem \ref{thm:baseweight} is again $A_p$-consistent for all $p\in (1, \infty)$.

Take $n \in \N$ and let for $i=1,\cdots,n$ the triple $(\Omega_i,\Sigma_i,\mu_i)$ be a $\sigma$-finite measure space. Define the product measure space
\begin{equation*}
  (\Omega,\Sigma,\mu)=(\Omega_1\times\cdots\times \Omega_n, \Sigma_1\times\cdots \times \Sigma_n,\mu_1\times\cdots\times\mu_n)
\end{equation*}
Then of course $(\Omega,\Sigma,\mu)$ is also $\sigma$-finite. For $\vect{q} \in (1,\infty)^n$ we write
\begin{equation}\label{eq:Lqit}
L^\vect{q}(\Omega)=L^{q_1}(\Omega_1,\cdots L^{q_n}(\Omega_n)).
\end{equation}

Next we extend Theorem \ref{thm:baseweight} to values in the above mixed $L^\vect{q}(\Omega)$ spaces. For the case $\Omega=\N$ this was already done in \cite[Corollary 3.12]{CMP}.
\begin{theorem}
\label{thm:ineq}
Let $f,g :\R^d\times \Omega \to \R_+$ be a pair of nonnegative, measurable functions and suppose that for some $p_0\in (1,\infty)$ there exists an increasing function $\alpha$ on $\R_+$ such that for all $w_0 \in A_{p_0}$
\begin{equation}
\label{eq:ass}
\nrm{ f(\lcdot,s)}_{L^{p_0}(\R^d,w_0)} \leq \alpha([w_0]_{A_{p_0}}) \nrm{ g(\lcdot,s)}_{L^{p_0}(\R^d,w_0)}
\end{equation}
for all $s \in \Omega$. Then for all $p \in (1,\infty)$ and $\vect{q} \in (1,\infty)^n$ there exist $c_{p,\vect{q},d}>0$ and $\beta_{p_0, p,\vect{q}}>0$ such that for all $w \in A_{p}$,
\begin{equation}
\label{eq:thm}
\nrm{f}_{L^p(\R^d,w;L^\vect{q}(\Omega))}\leq 4^n \alpha\Big(c_{p,\vect{q},d}[w]_{A_{p}}^{\beta_{p_0, p,\vect{q}}}\Big)  \nrm{g}_{L^p(\R^d,w;L^\vect{q}(\Omega))}.
\end{equation}
\end{theorem}

\begin{proof}
We will prove this theorem by induction. The base case $n=0$ is just weighted extrapolation, as covered in Theorem \ref{thm:baseweight}.

Now take $n \in \N\cup\{0\}$ arbitrary and assume that the assertion holds for all pairs $f,g:\R^d\times \Omega \to \R_+$ of nonnegative, measurable functions. Let $(\Omega_0,\Sigma_0,\mu_0)$ be a $\sigma$-finite measure space and take nonnegative, measurable functions $f,g: \R^d\times \Omega_0 \times \Omega \to \R_+$. Assume that \eqref{eq:ass} holds for $p_0$, all $w \in A_{p_0}$ and all $s \in \Omega_0\times \Omega$.

Now take $(s_0,s_1,\cdots,s_{n})\in \Omega_0 \times \Omega$ arbitrary. Let $\vect{q} \in (1,\infty)^{n}$ be given and take $r \in (1,\infty)$ arbitrary. Define $\vect{r}=(r,q_1,\cdots,q_n)$ and the pair of functions $F,G:\R^d \to [0,\infty]$ as
\begin{equation*}
F(x)=\nrm*{f(x,\lcdot)}_{L^\vect{r}(\Omega\times \Omega_0)} \hspace{1cm} G(x)=\nrm*{g(x,\lcdot)}_{L^\vect{r}(\Omega\times \Omega_0)}
\end{equation*}
By our induction hypothesis we know for all $p \in (1,\infty)$ there exist $c_{p,\vect{q},d}$ and $\beta_{p_0, p,\vect{q}}$ such that for all $w \in A_{p}$
 \begin{equation*}
   \nrm{f(\lcdot,s_0,\lcdot)}_{L^p(\R^d,w;L^{\vect{q}}(\Omega))} \leq 4^n \alpha(c_{p,\vect{q},d}[w]_{A_{p}}^{\beta_{p_0, p,\vect{q}}}) \nrm{g(\lcdot,s_0,\lcdot)}_{L^p(\R^d,w;L^{\vect{q}}(\Omega))}
 \end{equation*}
Now taking $p=r$ we obtain
\begin{align*}
\nrm{F}_{L^{r}(\R^d,w)}&=\Big(\int_{\Omega_{0}}\int_{\R^d} \nrm{f(x,s_0,\lcdot)}_{L^\vect{q}(\Omega)}^{r}w(x) \dd x\dd \mu_{0}\Big)^{\frac1r}\\
&\leq 4^n \alpha(c_{r,\vect{q},d}[w]_{A_{r}}^{\beta_{p_0, r,\vect{q}}}) \Big(\int_{\Omega_{0}}\int_{\R^d} \nrm{g(x,s_0,\lcdot)}_{L^\vect{q}(\Omega)}^{r}w(x) \dd x\dd \mu_{0}\Big)^{\frac1r}
\\ &=4^n  \alpha(c_{r,\vect{q},d}[w]_{A_{r}}^{\beta_{p_0, r,\vect{q}}})  \nrm{G}_{L^{r}(\R^d,w)}
\end{align*}
using Fubini's theorem in the first and third step. So with Theorem \ref{thm:baseweight} using $p_0=r$ we obtain for all $p \in (1,\infty)$ that there exist $c_{r,p,\vect{q},d}>0$ and $\beta_{p_0,p,\vect{r}}>0$ such that for all $w \in A_{p}$,
\begin{align*}
\nrm{f}_{L^p(\R^{d},w;L^\vect{r}(\Omega_0\times \Omega))} & =\nrm{F}_{L^p(\R^{d},w)}\\ & \leq 4^{n+1}  \alpha\Big(c_{r,p,\vect{q},d} [w]_{A_{p}}^{\beta_{p_0,p,\vect{r}}}\Big) \nrm{G}_{L^{p}(\R^d,w)}
\\ & =4^{n+1}  \alpha\Big((c_{r,p,\vect{q},d} [w]_{A_{p}}^{\beta_{p_0,p,\vect{r}}}\Big))\nrm{g}_{L^p(\R^d,w;L^\vect{r}(\Omega_0\times \Omega))}.
\end{align*}
This proves \eqref{eq:thm} for $n+1$.
\end{proof}

\begin{remark}\label{rem:RubioT}
Note that in the application of Theorem \ref{thm:ineq} it will often be necessary to use an approximation by simple functions to check the requirements, since point evaluations in \eqref{eq:ass} are not possible in general. Furthermore note that in the case that $f=Tg$ with $T$ a bounded linear operator on $L^p(\R^d,w)$ for all $w \in A_{p}$ this theorem holds for all UMD Banach function spaces, which is one of the deep results of Rubio de Francia and can be found in \cite[Theorem 5]{RubioD}.
\end{remark}

As an application of Theorem \ref{thm:ineq} we will present a short proof of the boundedness of the Hardy-Littlewood maximal operator on mixed $L^{\vect{q}}$-spaces.

\begin{definition}
Let $p \in (1,\infty)$ and $w \in A_p$. For $f \in L^p(\R^d,w;X)$ with $X = L^{\overline{q}}(\Omega)$ we define the maximal function $\widetilde{M}$ as
\begin{equation*}
\widetilde{M}f(x,s)=\sup_{Q \ni x} \avint_{Q} \abs{f(y,s)}\dd y
\end{equation*}
with $Q$ all cubes in $\R^d$ as before.
\end{definition}

We can see that $\widetilde{M}$ is measurable, as the value of the supremum in the definition stays the same if we only consider rational cubes.
We will show that the maximal function is bounded on the space $X = L^{\overline{q}}(\Omega)$. Note that if $\Omega = \N$, the result below reduces to the weighted version of the Fefferman-Stein theorem \cite{AJ80}.
\begin{theorem}\label{thm:maxbdd}
$\widetilde{M}$ is bounded on $L^p(\R^d,w;L^{\overline{q}}(\Omega))$ for all $p \in (1,\infty)$ and $w \in A_p$.
\end{theorem}

\begin{proof}
 Let $M$ be the Hardy-Littlewood maximal operator and assume that $f \in L^p(\R^d,w;L^{\overline{q}}(\Omega))$ is simple. By Proposition \ref{prop:propweight} and the definition of the Hardy-Littlewood maximal operator we know that
  \begin{equation*}
    \nrm{\widetilde{M}f(\lcdot,s)}_{L^p(\R^d,w)}=    \nrm{Mf(\lcdot,s)}_{L^p(\R^d,w)}\leq C_{p,d} \cdot [w]_{A_{p}}^{\frac{1}{p-1}}\nrm{f(\lcdot,s)}_{L^p(\R^d,w)}
  \end{equation*}
Then by Theorem \ref{thm:ineq} we get that
\begin{equation*}
  \nrm{\widetilde{M}f}_{L^p(\R^d,w;L^{\overline{q}}(\Omega))} \leq \alpha_{p,\vect{q},d}([w]_{A_{p}})\nrm{f}_{L^p(\R^d,w;L^{\overline{q}}(\Omega))}
\end{equation*}
with $\alpha_{p,\vect{q},d}$ an increasing function on $\R_+$. With a density argument we then get that $\widetilde{M}$ is bounded on $L^p(\R^d,w;L^{\overline{q}}(\Omega))$.
\end{proof}

\begin{remark}
Using deep connections between harmonic analysis with weights and martingale theory, Theorem \ref{thm:maxbdd} was obtained in \cite{Bourgain-ext} and
\cite[Theorem 3]{RubioD} for UMD Banach function spaces in the case $w=1$. It has been extended to the weighted setting in \cite{Toz}. As our main result Theorem \ref{thm:mainnew} is formulated for iterated $L^{\overline{q}}(\Omega)$-spaces we prefer the above more elementary treatment.
\end{remark}

\section{Main result}

In this section we present the proofs of Theorems \ref{thm:intro} and \ref{thm:mainnew} and Corollary \ref{cor:weightedells} which are our main results. In Subsection \ref{subs:ell} we will first obtain a preliminary result which is one of the ingredients in the proofs.

\subsection{$\ell^s$-boundedness\label{subs:ell}}

In this section we will introduce $\ell^s$-boundedness and present some simple examples. For this we will use the notion of a Banach lattice (see \cite{LiTz}). An example of a Banach lattice is $L^p$ or any Banach function space (see \cite[Section 63]{Zaanen}). In our main results only iterated $L^p$-spaces will be needed.

Although $\ell^s$-boundedness is used implicitly in the literature for operators on $L^p$-spaces, on Banach functions spaces it was introduced in \cite{Weis01a} under the name $\Rr_s$-boundedness. An extensive study can be found in \cite{KunstUll,UllmannPhD}.
\begin{definition}\label{def:ells}
Let $X$ and $Y$ be Banach lattices  and let $s \in [1,\infty]$. Then we call a family of operators $\T \subseteq \B(X,Y)$ $\ell^s$-bounded if there exists a constant $C$ such that for all integers $N$, for all sequences $(T_n)_{n=1}^ N$ in $\T$ and $(x_n)_{n=1}^N$ in $X$,
\begin{equation*}
\Big\|\ha*{\sum_{n=1}^N\abs*{T_n x_n}^ s}^{\frac{1}{s}}\Big\|_Y\leq C\Big\|\ha*{\sum_{n=1}^N\abs*{x_n}^ s}^{\frac{1}{s}}\Big\|_X
\end{equation*}
with the obvious modification for $s=\infty$. The least possible constant $C$ is called the {\em $\ell^{s}$-bound} of $\T$ and is
denoted by $\Rr^{\ell^s}(\T)$ and often abbreviated as $\Rr^{s}(\T)$.
\end{definition}

\begin{example}
\label{ex:pq}
Take $p \in (1,\infty)$ and let $\T\subseteq \B(L^p(\R^d))$ be uniformly bounded by a constant $C$. Then $\T$ is $\ell^p$-bounded with $\Rr^{p}(\T)\leq C$.
\end{example}

The following basic properties will be needed later on.
\begin{proposition}\label{prop:inters}
Let $\T\subseteq{\calL(X,Y)}$, where $X$ and $Y$ are Banach function spaces.
\begin{enumerate}
\item\label{it:interp} Let $1\leq s_0<s_1\leq \infty$ and assume that $X$ and $Y$ have an order continuous norm.
If $\T\subseteq{\calL(X,Y)}$ is $\ell^{s_j}$-bounded for $j=0,1$, then $\T$ is $\ell^s$-bounded for all $s \in [s_0,s_1]$ and with $\theta = \frac{s-s_0}{s_1-s_0}$, the following estimate holds:
\[\Rr^{s}(\T) \leq \Rr^{s_0}(\T)^{1-\theta} \Rr^{s_1}(\T)^{\theta}\leq \max\{\Rr^{s_0}(\T),\Rr^{s_1}(\T)\}\]
\item\label{it:dual} If $\T$ is $\ell^s$-bounded, then the adjoint family $\T^* = \{T^*\in \calL(Y^*, X^*): T\in \T\}$ is $\ell^{s'}$-bounded and $\Rr^{s'}(\T^*) = \Rr^{s}(\T)$.
\end{enumerate}
\end{proposition}

\begin{proof}
\eqref{it:interp} follows from Calder\'on's theory of complex interpolation of vector-valued function spaces (see \cite{Cal64} and \cite[Proposition 2.14]{KunstUll}). For \eqref{it:dual} we refer to \cite[Proposition 2.17]{KunstUll} and \cite[Proposition 3.4]{NVW}.
\end{proof}

\begin{remark}
Below we will only need Proposition \ref{prop:inters} in the case $X = Y=L^{\overline{q}}(\Omega)$. To give the details of the proof of Proposition \ref{prop:inters} in this situation one first needs to know that $X^* = L^{\overline{q}'}(\Omega)$ which can be obtained by elementary arguments (see Proposition \ref{prop:duality} below).  As a second step one needs to show that $X(\ell^s_N)^* = X^*(\ell^{s'}_N)$ and this is done in Lemma \ref{lem:nietref}.
\end{remark}

\begin{example}\label{ex:estimatesforells}
Let $1\leq s_0\leq q\leq s_1\leq \infty$. Let $X = L^q(\Omega)$ and let $\T\subset \calL(X)$ be $\ell^{s_j}$-bounded for $j\in \{0,1\}$. Then for $s\in [s_0, q]$, $\Rr^s(\T)\leq \Rr^{s_0}(\T)$ and for $s\in [q, s_1]$,  $\Rr^s(\T)\leq \Rr^{s_1}(\T)$. Indeed, note that by Example \ref{ex:pq},
\[\Rr^q(\T) = \sup_{T\in \T} \|T\| \leq \Rr^{s_j}(\T), \ \ \ j\in \{0,1\}.\]
Now the estimates follow from Proposition \ref{prop:inters} by interpolating with exponents $(s_0, q)$ and $(q,s_1)$.

In particular, it follows that the function $s \mapsto \Rr^s(\T)$, is decreasing on $[s_0, q]$ and increasing on $[q, s_1]$.
\end{example}

\subsection{Convolution operators}
Let $\K$ be the following class of kernels
\[\K  =   \{k\in L^1(\R^d): \text{for all simple} \ f:\R^d\to \R_+ \ \text{one has} \  |k| * f\leq Mf \ \text{a.e.}\}.\]
There are many examples of classes of functions $k$ with this property (see \cite[Chapter 2]{GrafakosC} and \cite[Proposition 4.5 and 4.6]{NVW}). It follows from \cite[Lemma 4.3]{NVW} that every $k \in \K$ satisfies $\nrm{k}_{L^1(\R^d)}\leq 1$.


To keep the presentation as simple as possible we only consider the iterated space $X = L^{\overline{q}}(\Omega)$ with $\vect{q}\in (1, \infty)^n$ below (see \eqref{eq:Lqit}).
For a kernel $k \in L^1(\R^d)$, $p \in (1,\infty)$ and $w \in A_p$ define the convolution operator $T_k$ on $L^p(\R^d,w;X)$  as $T_k f=k*f$.
Of course by the definition of $\widetilde{M}$ we also have
$\abs{k * f}\leq \widetilde{M}f$ almost everywhere for all simple $f:\R^d \to X$.

\begin{proposition}
\label{prop:lbdd1}
Let $\vect{q}\in (1, \infty)^n$ and $X = L^{\overline{q}}(\Omega)$. For all $s\in [1, \infty]$ and $p\in (1, \infty)$ and $w\in A_p$, the family of convolution operators $\T=\{T_k:k \in \K\}$ on $L^p(\R^d,w;X)$ is $\ell^s$-bounded and there is an increasing function $\alpha_{p,\vect{q},s,d}$ such that $\Rr^s(\T)\leq \alpha_{p,\vect{q},s,d}([w]_{A_{p}})$.
\end{proposition}
\begin{proof}
Let $1<s<\infty$. Assume that $f_1,\cdots,f_N$ are simple. Take $t \in \Omega$ and $i \in \{1,\cdots,N\}$ arbitrary. Note that we have $f_i(\lcdot,t) \in L^p(\R^d,w)$.
Then since $\abs{T_{k_i} f_i(x,t)} \leq \widetilde{M} f_i (x,t)$ for almost all $x \in \R^d$, the result follows from Theorem \ref{thm:maxbdd} using the vector $(q_1,\cdots,q_n,s)$ and the measure space
\begin{equation*}
\ha*{\Omega \times\{1,\cdots,N\}, \Sigma \times P(\{1,\cdots,N\}),\mu \times \lambda}
\end{equation*}
with $\lambda$ the counting measure. Now the result follows by the density of the simple functions in $L^{p}(\R^d,w;L^{\overline{q}}(\Omega))$.

The proof of the cases $s=1$ and $s=\infty$ follow the lines of \cite[Theorem 4.7]{NVW}, where the unweighted setting is considered. In the case $s=\infty$ also assume that $f_1,\cdots,f_N$ are simple. With the boundedness of $\widetilde{M}$ from  Theorem \ref{thm:maxbdd} we have
\begin{align*}
  \int_{\R^d}\Big\|\sup_{1\leq n \leq N}& \abs{T_{k_n}f_n(x)}\Big\|^p_{L^{\overline{q}}(\Omega)}w(x)\dd x \leq \int_{\R^d}\nrm*{\sup_{1\leq n \leq N}\widetilde{M}f_n(x)}^p_{L^{\overline{q}}(\Omega)}w(x)\dd x\\
  &\leq \int_{\R^d}\nrm*{\widetilde{M}\ha*{\sup_{1\leq n \leq N}\abs{f_n}}(x)}^p_{L^{\overline{q}}(\Omega)}w(x)\dd x\\
  &\leq \alpha_{p,\vect{q},d}( [w]_{A_{p}})^p \int_{\R^d}\nrm*{\ha*{\sup_{1\leq n \leq N}\abs{f_n}}(x)}^p_{L^{\overline{q}}(\Omega)}w(x)\dd x
\end{align*}
with $\alpha_{p,\vect{q},d}$ an increasing function on $\R_+$. The claim now follows by the density of the simple functions in $L^{p}(\R^d,w;L^{\overline{q}}(\Omega))$.

For $s=1$ we use duality. For $f\in L^{p}(\R^d,w;X)$ and $g\in L^{p'}(\R^d,w';X^*)$, let
\[\lb f, g\rb = \int_{\R^d} \lb f(x), g(x)\rb_{X,X^*} \, dx.\]
It follows from Proposition \ref{prop:duality} that in this way $L^{p}(\R^d,w;X)^* = L^{p'}(\R^d,w';X^*)$.
Moreover, one has $T_k^* = T_{\tilde{k}}$ with $\tilde{k}(x)=k(-x)$. Now since $k \in \K$ if and only if $\tilde{k} \in \K$ we know by the second case that the adjoint family $\T^*= \{T^*:T\in \T\}$ is $\ell^\infty$-bounded on $L^{p'}(\R^d,w';X^*)$. Now the result follows from Proposition \ref{prop:inters}.
\end{proof}


\begin{remark}
Proposition \ref{prop:lbdd1} is an extension of \cite[Theorem 4.7]{NVW} to the weighted setting.
The result remains true for UMD Banach function spaces $X$ and can be proved using the same techniques of \cite{NVW} where one needs to apply the weighted extension of \cite[Theorem 3]{RubioD} which is obtained in \cite{Toz}.

The endpoint case $s=1$ of Proposition \ref{prop:lbdd1} plays a crucial r\^ole in the proof of Theorems \ref{thm:intro} and \ref{thm:mainnew}.
Quite surprisingly the case $s=1$ plays a central r\^ole in the proof of \cite[Theorem 7.2]{NVW} as well, where it is used to prove $\Rr$-boundedness of a family of stochastic convolution operators.
\end{remark}

\subsection{Integral operators with operator valued kernel}

In this section $(\Omega,\Sigma,\mu)$ is a $\sigma$-finite measure space such that $L^q(\Omega)$ is separable for some (for all) $q\in (1, \infty)$.

\begin{definition}\label{def:ITsq}
Let $\J$ be an index set. For each $j\in \J$, let $T_j:\R^d\times \R^d \to \calL(L^q(\Omega))$ be such that
for all $\phi \in L^q(\Omega)$, $(x,y) \mapsto T_j(x,y) \phi$ is measurable and $\|T_j(x,y)\|\leq 1$.
For $k \in \K$ define the operator $I_{k,T_j}$ on $L^p(\R^d,v;L^q(\Omega))$ as
\begin{equation}\label{eq:IkTnieuw}
I_{k,T_j} f(x)=\int_{\R^d} k(x-y)T_j(x,y)f(y)\dd y
\end{equation}
and denote the family of all such operators by $\I_{T}$.
\end{definition}
In the above definition we consider a slight generalization of the setting of Theorem \ref{thm:intro}: We allow different operators $T_j$ for $j\in \J$ in the $\ell^s$-boundedness result of Theorem \ref{thm:mainnew}.

We first prove that the family of operators $\I_{T}$ is uniformly bounded.
\begin{lemma}
\label{thm:bddI}
Let $1<p,q<\infty$ and write $X = L^q(\Omega)$. Assume that for all $\phi\in X$ and $j\in \J$, $(x,y)\mapsto T_j(x,y)\phi$ is measurable and $\|T_j(x,y)\|\leq 1$. Then there exists an increasing function $\alpha_{p,d}$ on $\R_+$ such that for all $I_{k,T_j}\in \I_{T}$,
\begin{equation*}
\nrm*{I_{k,T_j}}_{\calL\ha*{L^p(\R^d,v;X)}}\leq \alpha_{p,d}([v]_{A_{p}}), \ \ \ v\in A_{p}.
\end{equation*}
\end{lemma}

\begin{proof}
Let $f \in L^p(\R^d,v;X)$ arbitrary. Then by Minkowski's inequality for integrals in $(i)$, the properties of $k \in \K$ in (ii) and boundedness of $M$ on $L^p(\R^d,v)$ in $(iii)$, we get
\begin{align*}
\|I_{k,T_j} &f\|_{L^p(\R^d,v;X)} = \ha*{\int_{\R^d} \nrm*{\int_{\R^d} k(x-y)T_j(x,y)f(y) \dd y}_{X}^p v(x) \dd x}^{\frac{1}{p}}\\
&\stackrel{(i)}{\leq}\ha*{ \int_{\R^d} \ha*{\int_{\R^d} |k(x-y)|\nrm{T_j(x,y)f(y)}_{X}\dd y}^p v(x) \dd x}^{\frac{1}{p}}\\
&\leq \ha*{ \int_{\R^d} \ha*{\int_{\R^d} |k(x-y)|\nrm{f(y)}_{X}\dd y}^p v(x) \dd x}^{\frac{1}{p}}\\
&\stackrel{(ii)}{\leq} \ha*{\int_{\R^d} \ha*{ M\ha{\nrm{f}_{X}}(x)}^p v(x) \dd x}^{\frac{1}{p}}
\stackrel{(iii)}{\leq} \alpha_{p,d}([v]_{A_{p}}) \, \nrm{ f}_{L^p(\R^d,v;X)}
\end{align*}
with $\alpha_{p,d}$ an increasing function on $\R_+$. This proves the lemma.
\end{proof}

\begin{theorem}
\label{thm:mainnew}
Let $1<p,q<\infty$ and write $X = L^q(\Omega)$. Assume the following conditions
\begin{enumerate}[(1)]
\item For all $\phi\in X$ and $j\in \J$, $(x,y)\mapsto T_j(x,y)\phi$ is measurable.
\item For all $s\in (1, \infty)$, $\T = \{T_j(x,y):x,y\in \R^d, j\in \J\}$ is $\ell^s$-bounded,
\end{enumerate}
Then for all $v\in A_p$ and all $s\in (1, \infty)$, the family of operators $\I_{T}\subseteq L^p(\R^d,v;X)$ as defined in \eqref{eq:IkTnieuw}, is $\ell^s$-bounded with $\Rr^s(\I_T) \leq C$ where $C$ depends on $p, q, d, s, [v]_{A_p}$ and on $\Rr^\sigma(\T)$ for $\sigma\in(1, \infty)$ and is $A_p$-consistent.
\end{theorem}

\begin{example}\label{ex:weighted}
When $\Omega = \R^e$ with $\mu$ the Lebesgue measure and $q_0\in (1, \infty)$, then the weighted boundedness of each of the operators $T_j(x,y)$ on $L^{q_0}(\R^e,w)$ for all $A_{q_0}$-weights $w$ in an $A_{q_0}$-consistent way, is a sufficient condition for the $\ell^s$-boundedness which is assumed in Theorem \ref{thm:mainnew}.
Indeed, this follows from \cite[Corollary 3.12]{CMP} (also see Theorem \ref{thm:ineq}).

Usually, the weighted boundedness is simple to check with \cite[Theorem IV.3.9]{Rubioboek} or
\cite[Theorem 9.4.6]{GrafakosM}, because often for each $x,y\in\R^d$ and $j\in \J$, $T_j(x,y)$ is given by a Fourier multiplier operator in $\R^e$.
\end{example}

\begin{example}
Let $q\in (1, \infty)$. Let $T(t)  = e^{t\Delta}$ for $t\geq 0$ be the heat semigroup, where $\Delta$ is the Laplace operator on $\R^e$. Then it follows from the weighted Mihlin multiplier theorem \cite[Theorem IV.3.9]{Rubioboek}) that for all $w\in A_{q}$, $\|T(t)\|_{\calL(L^q(\R^e,w))}\leq C$, where $C$ is $A_q$-consistent. Therefore, as in Example \ref{ex:weighted}, $\{T(t):t\in \R_+\}$ is $\ell^s$-bounded on $L^q(\R^d,w)$ by an $A_q$-consistent $\Rr^s$-bound.

In order to give an example of an operator $I_{k,T}$ as in \eqref{eq:IkTnieuw}, we could let $T(x,y) = T(\phi(x,y))$, where $\phi:\R^d\times \R^d\to \R_+$ is measurable. Other examples can be given if one replaces the heat semigroup by a two parameter evolution family $T(t,s)$. As explained in the introduction, this is the setting of \cite{GV} (see Theorem \ref{thm:intro}).
\end{example}

To prove Theorem \ref{thm:mainnew} we will first show a result assuming $\ell^s$-boundedness for a fixed $s\in (1, \infty)$. Here we can also include $s=1$.
\begin{proposition}
\label{prop:mainprel}
Let $1\leq s<q<\infty$ and write $X = L^q(\Omega)$. Assume the following conditions
\begin{enumerate}[(1)]
\item For all $\phi\in X$ and all $j\in \J$, $(x,y)\mapsto T_j(x,y)\phi$ is measurable.
\item $\T = \{T_j(x,y):x,y\in \R^d, j\in \J\}$ is $\ell^s$-bounded.
\end{enumerate}
Then for all $p\in (s, \infty)$ and all  $v \in A_{\frac{p}{s}}$ the family of operators $\I_{T}\subseteq L^p(\R^d,v;X)$ defined as in \eqref{eq:IkTnieuw}, is $\ell^s$-bounded and there exist an increasing function $\alpha_{s,p,q,d}$ such that
\[\Rr^s(\I_{T})\leq \Rr^s(\T) \alpha_{s,p,q,d}([v]_{A_{\frac{p}{s}}}).\]
\end{proposition}
\begin{proof}
Without loss of generality we can assume $\Rr^s(\T) = 1$.
We start with a preliminary observation. By \cite[Theorem VI.5.2]{Rubioboek} or \cite[Theorem 9.5.8]{GrafakosM}, the $\ell^s$-boundedness is equivalent to the following: for every $u\geq 0$ in $L^{\frac{q}{q-s}}(\Omega)$ there exists a $U\in L^{\frac{q}{q-s}}(\Omega)$ such that
\begin{equation}\label{eq:factorizationLs}
\begin{aligned}
\|U\|_{L^{\frac{q}{q-s}}(\Omega)} &\leq \|u\|_{L^{\frac{q}{q-s}}(\Omega)},
\\ \int_\Omega |T_j(x,y)\phi|^s u \, d\mu& \leq \int_\Omega |\phi|^s U\, d\mu, \  \  x,y\in \Omega, \ j\in \J \ \phi\in L^q(\Omega).
\end{aligned}
\end{equation}

For $n=1,\cdots,N$ take $I_{k_n,T_{j_n}}\in \I_{T}$ and let $I_n = I_{k_n,T_{j_n}}$ where $j_1, \ldots, j_N\in \J$. Take $f_1,\cdots,f_N \in L^p(\R^d,v;X)$ and note that
\begin{equation*}
  \nrmn{\ha*{\sum_{n=1}^N\abs*{I_{n} f_n}^s}^{\frac{1}{s}}}_{L^p(\R^d,v;X)}= \nrmn{\sum_{n=1}^N\abs*{I_{n} f_n}^s}_{L^{\frac{p}{s}}\ha*{\R^d,v;L^{\frac{q}{s}}(\Omega)}}^{\frac{1}{s}}.
\end{equation*}
Let $r \in (1,\infty)$ be such that $\frac{1}{r}+\frac{s}{q}=1$ and fix $x \in \R^d$. As $L^r(\Omega) = L^{\frac{q}{s}}(\Omega)^*$, we can find a function $u \in L^r(\Omega)$, which will depend on $x$, with $u \geq 0$ and $\nrm{u}_{L^r(\Omega)} = 1$ such that
\begin{equation}
\label{eq:nrmdual}
  \nrmn{\sum_{n=1}^N\abs*{I_n f_n(x)}^s}_{L^{\frac{q}{s}}(\Omega)}=\sum_{n=1}^N\int_{\Omega}\abs*{I_n f_n(x)}^s u \dd \mu.
\end{equation}
By the observation in the beginning of the proof, there is a function $U\geq 0$ in $L^r(\Omega)$ (which depends on $x$ again) such that \eqref{eq:factorizationLs} holds.
Since $\nrm{k_n}_{L^1(\R^d)}\leq 1$, H\"older's inequality yields
\begin{equation}\label{eq:jensen}
  \abs{I_n f_n(x)}^s\leq \int_{\R^d}  |k_n(x-y)|\abs{T_{j_n}(x,y)f_n(y)}^s\dd y.
\end{equation}
Applying \eqref{eq:jensen} in $(i)$, estimate \eqref{eq:factorizationLs} in $(ii)$, and H\"older's inequality in $(iii)$, we get:
\begin{align*}
\sum_{n=1}^N \int_{\Omega}\abs*{I_nf_n(x)}^s u \dd \mu
& \stackrel{(i)}{\leq} \sum_{n=1}^N \int_{\Omega} \int_{\R^d} |k_n(x-y)| \abs{T_{j_n}(x,y)f_n(y)}^s\dd y \,  u \dd \mu\\
&=\sum_{n=1}^N\int_{\R^d} |k_n(x-y)| \int_{\Omega}\abs{T_{j_n}(x,y)f_n(y)}^s \, u \dd \mu\dd y\\
&\stackrel{(ii)}{\leq} \sum_{n=1}^N\int_{\R^d} |k_n(x-y)| \int_{\Omega}\abs{f_n(y)}^s \, U \dd \mu\dd y\\
&= \int_{\Omega} \sum_{n=1}^N\int_{\R^d} |k_n(x-y)|\abs{f_n(y)}^s \dd y \, U \dd \mu\\
&\stackrel{(iii)}{\leq} \nrmn{\sum_{n=1}^N\int_{\R^d} |k_n(x-y)|\abs{f_n(y)}^s\dd y}_{L^{\frac{q}{s}}(\Omega)}.
\end{align*}
Combining \eqref{eq:nrmdual} with the above estimate and applying the $\ell^1$-boundedness result of Proposition \ref{prop:lbdd1} to $\abs{f_n}^s \in {L^{\frac{p}{s}}\ha*{\R^d,v;L^{\frac{q}{s}}(\Omega)}}$ (here we use $v \in A_{\frac{p}{s}}$), we get
\begin{align*}
 \Big\| \Big(\sum_{n=1}^N\abs*{I_n f_n}^s\big)^{\frac{1}{s}}\Big\|_{L^p(\R^d,v;X)}  & \leq \nrmn{\sum_{n=1}^N\int_{\R^d} |k_n(\lcdot-y)|\abs{f_n(y)}^s\dd y}_{L^{\frac{p}{s}}\ha*{\R^d,v;L^{\frac{q}{s}}(\Omega)}}^{\frac{1}{s}}
 \\ &\leq \alpha_{p,q,s,d}([v]_{A_{\frac{p}{s}}})
  \nrmn{\sum_{n=1}^N \abs{f_n}^s}_{L^{\frac{p}{s}}\ha*{\R^d,v;L^{\frac{q}{s}}(\Omega^e,w)}}^{\frac{1}{s}}\\
 & = \alpha_{p,q,s,d}([v]_{A_{\frac{p}{s}}}) \nrmn{\ha*{\sum_{n=1}^N \abs{f_n}^s}^{\frac{1}{s}}}_{L^{p}\ha*{\R^d,v;X}}
\end{align*}
with $\alpha_{p,q,s,d}$ an increasing function on $\R_+$. This proves the $\ell^s$-boundedness.
\end{proof}

Next we prove Theorem \ref{thm:mainnew}. For a constant $\phi$ depending on a parameter $t\in I\subset \R$ , we write $\phi \propto t$ if $\phi_t\leq \phi_s$ whenever $t\leq s$ and $s,t\in I$.

\begin{proof}[Proof of Theorem \ref{thm:mainnew}]
Fix $q\in (1, \infty)$, $p=q$, $v\in A_q$ and $\kappa = 2[v]_{A_q}\geq 2$. The case $p\neq q$ will be considered at the end of the proof.

\textbf{Step 1.} First we prove the theorem for very small $s\in (1,q)$.
Proposition \ref{prop:propweight} gives $\sigma_1 = \sigma_{q,\kappa,d} \in (1, q)$ and $C_{q,\kappa,d}$ such that for all $s\in (1, \sigma_1]$ and all weights $u\in A_{q}$ with $[u]_{A_{q}}\leq \kappa$,
\[[u]_{A_{\frac{q}{s}}}\leq [u]_{A_{\frac{q}{\sigma}}}\leq C_{q,\kappa,d}.\]
Moreover, $\sigma_1 \propto \kappa^{-1}$ and $C\propto \kappa$.

By Proposition \ref{prop:mainprel}, $\I_{T}\subseteq \calL(L^q(\R^d,v;X))$ is $\ell^s$-bounded for all $s\in (1, \sigma_1)$ and
\begin{equation}\label{eq:Rsboundps}
\Rr^s(\I_{T})\leq \Rr^s(\T) \alpha_{s,q,d}([v]_{A_{\frac{q}{s}}})\leq \Rr^s(\T) \beta_{q,s,d,\kappa},
\end{equation}
with $\beta_{q,s,d,\kappa} = \alpha_{q,s,d}(C_{q,\kappa,d})$. Note that $\beta\propto \kappa$ and $\beta \propto s'$.

\textbf{Step 2.} Now we use a duality argument to prove the theorem for large $s\in (q, \infty)$.
By Proposition \ref{prop:propweight}, $v'\in A_{q'}$ and $\tilde{\kappa} = 2[v']_{A_{q'}} = 2[v]_{A_q}^{\frac{1}{q-1}} = 2(\kappa^{\frac{1}{q-1}})$.
Note that we can identify $X^*=L^{q'}(\Omega)$ and $L^q(\R^d,v;X)^*=L^{q'}(\R^d, v';X^*)$ by Proposition \ref{prop:duality}.
Define $\I_{T}^*=\{I^*:I \in \I_{T}\}$.

It is standard to check that for $I_{k,T_j} \in \I_{T}$ the adjoint $I_{k,T_j}^*$ satisfies
\begin{equation*}
I_{k,T_j}^* g(x)=\int_{\R^d} \tilde{k}(y-x)\tilde{T}_j(x,y)g(y)\dd y=I_{\tilde{k},\tilde{T}_j}g(x)
\end{equation*}
with $\tilde{k}(x)=k(-x)$ and $\tilde{T}_j(x,y)=T^*_j(y,x)$. As already noted before we have $\tilde{k}\in \K$.
Furthermore, by Proposition \ref{prop:inters} the adjoint family $\T^*$ is $\Rr^{s'}$-bounded with $\Rr^{s'}(\T^*) = \Rr^{s}(\T)$.
Therefore, it follows from Step 1 that there is a $\sigma_2 = \sigma_{q',\tilde{\kappa},d}\in (1, q')$
such that for all $s'\in (1, \sigma_2]$, $\I_{T}^*$ is $\ell^{s'}$-bounded on $L^{q'}(\R^d, v';X^*)$ and
using Proposition \ref{prop:inters} again, we obtain $\I_T$ is $\ell^s$-bounded and
\begin{equation}\label{eq:dualityIest}
\Rr^{s}(\I_T) = \Rr^{s'}(\I_T^*)\leq \Rr^{s'}(\T^*) \beta_{q',s',d,\tilde{\kappa}} = \Rr^{s}(\T) \beta_{q',s',d,\tilde{\kappa}}.
\end{equation}
Therefore, Proposition \ref{prop:inters} yields that $\I_T$ is $\ell^s$-bounded on $L^{q}(\R^d, v;X)$ for all $s\in [\sigma_2', \infty)$.

\textbf{Step 3.} We can now finish the proof in the case $p=q$ by an interpolation argument. In the previous steps 1 and 2 we have found $1<\sigma_1<q<\sigma_2'<\infty$ such that $\I_\alpha$ is $\ell^s$-bounded for all $s \in (1,\sigma_1]\cup[\sigma_2',\infty)$
with
\begin{equation}\label{eq:Rsgamma}
\Rr^s(\I_T)\leq \Rr^{s}(\T) \gamma_{q,s,d,\kappa}.
\end{equation}
where $\gamma_{q,s,d,\kappa} =  \beta_{q,s,d,\kappa}$ if $s\leq \sigma_1$ and $\gamma_{q,s,d,\kappa} = \beta_{q',s',d,\tilde{\kappa}}$ if $s\geq \sigma_2'$. Clearly, $\gamma := \gamma_{q,s,d,\kappa}$ satisfies  $\gamma\propto \kappa$, $\gamma \propto s'$ for $s\in (1, \sigma_1]$ and $\gamma \propto s$ for $s\in [\sigma_2', \infty)$. Moreover, $\sigma_1 \propto \frac{1}{\kappa}$  and $\sigma_2' \propto \kappa$.

Now Proposition \ref{prop:inters} yields the $\ell^s$-boundedness and the required estimates for the remaining $s\in [\sigma_1,\sigma_2']$ and by \eqref{eq:Rsgamma} we find
\begin{align*}
\Rr^s(\I_T)& \leq \max\{\Rr^{\sigma_1}(\I_T), \Rr^{\sigma_2'}(\I_T)\}
\\ & \leq \max\{\Rr^{\sigma_1}(\T),\Rr^{\sigma_2'}(\T)\} \gamma.
\end{align*}
where $\gamma = \max\{\gamma_{q,\sigma_1,d,\kappa},\gamma_{q',\sigma_2,d,\tilde{\kappa}}\}$.
By Example \ref{ex:estimatesforells}, $\Rr^{\sigma_1}(\T)\propto \kappa$ and $\Rr^{\sigma_2'}(\T)\propto \kappa$.
Also $\gamma\propto \kappa$ in the above. Therefore, the obtained $\Rr^s$-bound is $A_q$-consistent.

\textbf{Step 4.} Next let $p,q\in (1, \infty)$. Fix $s\in (1, \infty)$. For $n=1,\cdots,N$ take $I_{k_n,T_{j_n}}\in \I_{T}$ and let $I_n = I_{k_n,T_{j_n}}$. Take $f_1,\cdots,f_N \in L^p(\R^d,v;X)\cap L^q(\R^d,v;X)$ and let
\[F = \Big\|\Big(\sum_{n=1}^N |I_n f_n|^s\Big)^{\frac1s}\Big\|_X \ \ \text{and} \ \ G = \Big\|\Big(\sum_{n=1}^N |f_n|^s\Big)^{\frac1s}\Big\|_X.\]
By the previous step we know that for all $v\in A_{q}$,
\begin{align*}
\|F\|_{L^q(\R^d,v)}\leq C \|G\|_{L^q(\R^d,v)},
\end{align*}
where $C$ depends on $d$, $s$, $q$, and $[v]_{A_p}$ and is $A_p$-consistent. Therefore, by Theorem \ref{thm:baseweight} we can extrapolate to obtain for all $p\in (1, \infty)$ and $v\in A_p$,
\begin{align*}
\|F\|_{L^p(\R^d,v)}\leq \tilde{C}\|G\|_{L^p(\R^d,v)},
\end{align*}
where $\tilde{C}$ depends on $C$, $p$ and $[v]_{A_p}$ and is again $A_p$-consistent. This implies the required $\Rr^s$-boundedness for all $p,q\in (1, \infty)$ with constant $\tilde{C}$.
\end{proof}

\begin{corollary}\label{cor:weightedells}
Let $\Omega\subseteq \R^d$ be an open set. Let $1<p,q,q_0<\infty$. Assume the following conditions
\begin{enumerate}[(1)]
\item For all $\phi\in L^q(\Omega)$ and $j\in \J$, $(x,y)\mapsto T_j(x,y)\phi$ is measurable.
\item For all $w\in A_{q_0}$, $\displaystyle \sup_{j\in \J, x,y\in \Omega}\|T_j(x,y)\|_{\calL(L^{q_0}(\Omega,w))}\leq C$,
where $C$ is $A_{q_0}$-consistent.
\end{enumerate}
Then for all $v\in A_p$ all $w\in A_q$ and all $s\in (1, \infty)$, the family of operators $\I_{T}\subseteq L^p(\R^d,v;L^q(\Omega,w))$ as defined in \eqref{eq:IkTnieuw}, is $\ell^s$-bounded with $\Rr^s(\I_T) \leq \tilde{C}$ where $\tilde{C}$ depends on $p, q, d,s,[v]_{A_p},[w]_{A_q}$ and on $\Rr^\sigma(\T)$ for $\sigma\in(1, \infty)$ and is $A_p$- and $A_q$-consistent.
\end{corollary}

\begin{proof}
In the case $\Omega = \R^e$, note that Example \ref{ex:weighted} yields that for each $q\in (1, \infty)$ and each $w\in A_q$ and $s\in (1, \infty)$, $\T$ considered on $L^q(\Omega,w)$ is $\ell^s$-bounded. Moreover, $\Rr^s(\T)\leq K$, where $K$ depends on $q, s, e$ and $[w]_{A_q}$ in an $A_q$-consistent way. Therefore, the result follows from Theorem \ref{thm:mainnew}.

In the case $\Omega\subseteq \R^e$, we reduce to the case $\R^e$ by a restriction-extension argument. For convenience we sketch the details. Let $E:L^q(\Omega,w)\to L^q(\R^e,w)$ be the extension by zero and let $R:L^q(\R^e,w)\to L^q(\Omega,w)$ be the restriction to $\Omega$.
For every $x,y\in \R^d$ and $j\in \J$, let $\tilde{T}_j(x,y) = E T_j(x,y) R\in \calL(L^q(\R^e,w))$ and let $\tilde{\T} = \{\tilde{T}_j(x,y):x,y\in \R^d\}$. Since $\|\tilde{T}_j(x,y)\|_{\calL(L^q(\R^e,w))}\leq \|T_j(x,y)\|_{\calL(L^q(\Omega,w))}\leq C$, it follows from the case $\Omega=\R^e$ that
$\I_{\tilde{T}}\subseteq L^p(\R^d,v;L^q(\R^e,w))$ is $\ell^s$-bounded with $\Rr^s(\I_{\tilde{T}}) \leq \tilde{C}$. Now it remains to observe that the restriction of $I_{k,\tilde{T}_j}$ to $L^p(\R^d,v;L^q(\Omega,w))$ is equal to $I_{k,T_j}$ and hence $\Rr^s(\I_{T})\leq  \Rr^s(\I_{\tilde{T}}) \leq \tilde{C}$.
\end{proof}

Next we will prove Theorem \ref{thm:intro}. In order to do so we recall the definition of $\Rr$-boundedness.
\begin{definition}
\label{def:Rbdd}
Let $X$ and $Y$ be Banach spaces and let $(\varepsilon_{n})_{n\geq 1}$ be a Rademacher sequence on a probability space $(A, \mathscr{A}, \mathbb{P})$. A family of operators $\scrS \subseteq \B(X,Y)$ is said to be {\em $\Rr$-bounded} if there exists a constant $C$ such that for all integers $N$, for all sequences $(S_n)_{n=1}^N$ in $\scrS$ and $(x_n)_{n=1}^N$ in $X$,
\begin{equation*}
\Big\|\sum_{n=1}^N \varepsilon_n S_n x_n \Big\|_{L^2(A;Y)}\leq C\Big\|\sum_{n=1}^N \varepsilon_n x_n \Big\|_{L^2(A;Y)}
\end{equation*}
The least possible constant $C$ is called the {\em $\Rr$-bound} of $\scrS$ and is
denoted by $\Rr(\scrS)$.
\end{definition}

\begin{remark}\label{rem:ell2R}
For $X = Y =  L^{\overline{q}}(\Omega)$ with $q\in (1, \infty)^n$, the notions $\ell^2$-boundedness and $\Rr$-boundedness of any family $\scrS\subseteq \B(X,Y)$ coincide and $C^{-1} \Rr^2(\scrS)\leq \Rr(\scrS)\leq C \Rr^2(\scrS)$, where $C$ is a constant which only depends on $\overline{q}$. This assertion follows from the Kahane-Khintchine inequalities (see \cite[1.10 and 11.1]{DJT}).
\end{remark}

\begin{proof}[Proof of Theorem \ref{thm:intro}]
The result follows directly from Corollary \ref{cor:weightedells} and
Remark \ref{rem:ell2R} with $X = L^p(\R;L^q(\Omega))$.
\end{proof}

\appendix

\section{Duality of iterated $L^{\overline{q}}$-spaces}

Let $(\Omega_i, \Sigma_i, \mu_i)$ for $i=1, \ldots n$ be $\sigma$-finite measure spaces.
The dual of the iterated space $L^{\vect{q}}(\Omega)$ as defined in \eqref{eq:Lqit}, is exactly what one would expect. In a general setting one can prove that $L^p(\Omega;X)^*=L^{p'}(\Omega,X^*)$ for reflexive Banach function spaces $X$ from which the duality for $L^{\vect{q}}(\Omega)$ follows, as is done in \cite[Chapter IV]{DiestelUhl} using the so-called Radon-Nikodym property of Banach spaces. Here we present an elementary proof just for $L^{\vect{q}}(\Omega)$.
\begin{proposition}
\label{prop:duality}
Let $\vect{q}\in (1, \infty)^n$. For every bounded linear functional $\Phi $ on $L^{\overline{q}}(\Omega)$ there exists a unique $g \in L^{\overline{q}'}(\Omega)$ such that:
\begin{equation}
\label{eq:goal}
\Phi(f)=\int_\Omega fg \dd \mu
\end{equation}
for all $f \in L^{\overline{q}}$ and $\nrm{\Phi}=\nrm{g}_{ L^{\overline{q}'}(\Omega)}$, i.e. $L^{\overline{q}}(\Omega)^*=L^{\overline{q}'}(\Omega)$.
\end{proposition}

\begin{proof}
We follow the strategy of proof from \cite[Theorem 6.16]{RudinRC}. The uniqueness proof is as in \cite[Theorem 6.16]{RudinRC}.
Also by repeatedly applying H\" older's inequality we have for any $g$ satisfying \eqref{eq:goal} that
\begin{equation}
\label{eq:equality}
\nrm{\Phi}\leq \nrm{g}_{L^{\overline{q}'}(\Omega)}.
\end{equation}
So it remains to prove that $g$ exists and that equality holds in \eqref{eq:equality}. As in \cite[Theorem 6.16]{RudinRC} one can reduce to the case $\mu(\Omega)<\infty$. Define $\lambda(E)=\Phi(\chi_E)$ for $E \in \Sigma$.
Then one can check that $\lambda$ is a complex measure which is absolutely continuous with respect to $\mu$. So by the Radon-Nikodym Theorem \cite[Theorem 6.10]{RudinRC} we can find a $g \in L^1(\Omega)$ such that for all measurable $E \subseteq \Omega$
\begin{equation*}
\Phi(\chi_E)=\int_Eg\dd \mu = \int_\Omega \chi_E g \dd \mu
\end{equation*}
and from this we get by linearity
$\Phi(f)= \int_\Omega f g \dd \mu$ for all simple functions $f$. Now take a $f \in L^\infty(\Omega)$ arbitrary and let $f_i$ be simple functions such that $\nrm{f_i-f}_{L^\infty(\Omega)}\to 0$ for $i \to \infty$. Then since $\mu(\Omega)<\infty$ we have $\nrm{f_i-f}_{L^{\overline{q}}(\Omega)}\to 0$ for $i \to \infty$. Hence
\begin{equation}
\label{eq:infty}
\Phi(f)=\lim_{i \to \infty} \Phi(f_i)=\lim_{i \to \infty} \int_\Omega f_i g \dd\mu= \int_\Omega f g \dd\mu.
\end{equation}

We will now prove that $g \in L^{\overline{q}'}(\Omega)$ and that equality holds in \eqref{eq:equality}. Take $k \in \N$ arbitrary. Let $E_k^1=\{s \in \Omega:\frac{1}{k} \leq \abs{g(s)}\leq k\}$ and define for $i=2,\cdots,n$
\begin{equation*}
  E_k^i=\cbrace*{s \in \Omega: \nrm*{g_k(s_1,\cdots,s_{i-1},\lcdot)}_{L^{q'_{i}}(\Omega_i,\cdots L^{q'_n}(\Omega_n))}\geq \frac{1}{k}}
\end{equation*}
Now take $g_k=g\prod_{i=1}^n\chi_{E^i_k}$ and let $\alpha$ be its complex sign function, i.e. $\abs{\alpha}=1$ and $\alpha \abs{g_k} = g_k$. Take
\begin{equation*}
f(s)=\overline{\alpha}\abs{g_k(s)}^{q_n'-1}\prod_{i=2}^{n}\nrm*{g_k(s_1,\cdots,s_{i-1},\lcdot)}^{q_{i-1}'-q_{i}'}_{L^{q'_{i}}(\Omega_i,\cdots L^{q'_n}(\Omega_n))}
\end{equation*}
where we define $0 \cdot \infty = 0$. Then $f \in L^\infty(\Omega)$ and one readily checks that
\begin{equation}
\label{eq:fg}
\begin{aligned}
\int_\Omega fg_k\dd \mu = \nrm{g_k}_{L^{\overline{q}'}(\Omega)}^{q_1'} \ \   \text{and} \ \  \nrm*{f}_{L^{\overline{q}}(\Omega)}&=\nrm{g_k}_{L^{\overline{q}'}(\Omega)}^{\frac{q_1'}{q_1}}.
\end{aligned}
\end{equation}
So from \eqref{eq:fg} we obtain
\begin{equation*}
\nrm{g_k}_{L^{\overline{q}'}(\Omega)}^{q_1'}=\int_\Omega fg_k\dd \mu = \Phi(f)\leq \nrm*{f}_{L^{\overline{q}}(\Omega)} \nrm{\Phi}= \nrm{g_k}_{L^{\overline{q}'}(\Omega)}^{\frac{q_1'}{q_1}} \nrm{\Phi}
\end{equation*}
which means $\nrm{g_k}_{L^{\overline{q}'}(\Omega)}\leq \nrm{\Phi}$. Since this holds for all $k \in \N$ we obtain by Fatou's lemma that $\nrm{g}_{L^{\overline{q}'}(\Omega)}\leq \nrm{\Phi}$, which proves that $g \in L^{\overline{q}'}(\Omega)$ and $\nrm{g}_{L^{\overline{q}'}(\Omega)}= \nrm{\Phi}$. From this we also get \eqref{eq:infty} for all $f \in {L^{\overline{q}'}(\Omega)}$ by H\"olders inequality and the dominated convergence theorem. This proves the required result.
\end{proof}

To obtain the duality result in Proposition \ref{prop:inters} for $s=1$ and $s=\infty$, one also needs the following end-point duality result.
Let $X(\ell^s_N)$ be the space of all $N$-tuples $(f_n)_{n=1}^N\in X^N$ with
\[\|(f_n)_{n=1}^N\|_{X(\ell^s_N)} = \Big\|\Big(\sum_{n=1}^N |f_n|^s\Big)^{1/s}\Big\|_X\]
with the usual modification if $s=\infty$.
\begin{lemma}
\label{lem:nietref}
  Define $X=L^{\overline{q}}(\Omega)$. Take $s \in [1,\infty]$ and $N \in \N$. Then for every bounded linear functional $\Phi$ on $X(\l^s_N)$ there exists a unique $g \in X^*(\l^{s'}_N)$ such that
  \begin{equation*}
    \Phi(f)=\sum_{i=1}^N\ip{f_i,g_i}_{X,X^*}
  \end{equation*}
  for all $f \in X(\l^s_N)$ and $\nrm{\Phi}=\nrm{g}_{X^*(\l^{s'}_N)}$, i.e.  $X(\l^{s}_N)^*=X^*(\l^{s'}_N)$.
\end{lemma}

Also this result can be proved with elementary arguments.
Indeed, for $r_1,r_2 \in [1, \infty]$ we have $X(\l^{r_1}_N) = X(\l^{r_2}_N)$ as sets and the following inequalities hold for all $f \in X(\l^{r}_N)$ and $r \in [1,\infty]$
  \begin{equation}
  \label{eq:eindigineq}
  \begin{aligned}
        \nrm{f}_{X(\l^{r}_N)}&\leq\nrm{f}_{X(\l^{1}_N)}\leq N^{1-\frac{1}{r}}\nrm{f}_{X(\l^{r}_N)}\\
        \nrm{f}_{X(\l^{\infty}_N)}&\leq\nrm{f}_{X(\l^{r}_N)}\leq N^{\frac{1}{r}}\nrm{f}_{X(\l^{\infty}_N)}.
  \end{aligned}
\end{equation}
Now the lemma readily follows from $X(\ell^r_N)^* = X^*(\ell^{r'}_N)$ for $r\in (1, \infty)$ and letting $r\downarrow 1$ and $r\uparrow \infty$.

\end{document}